\documentclass[11pt]{article}

\usepackage[utf8x]{inputenc}
\usepackage[english]{babel}
\usepackage[T1]{fontenc}
\usepackage{amsmath,amscd}
\usepackage{amsfonts}
\usepackage{amssymb}
\usepackage{amsthm}
\usepackage{mathrsfs}
\usepackage{graphicx}
\usepackage[all,cmtip]{xy}
\usepackage{textcomp}
\usepackage{url}
\usepackage{amsopn,latexsym,amscd}
\usepackage{color}
\usepackage{todonotes}
\usepackage{fullpage}
\usepackage{amsmath,graphicx}

\newtheorem{teo}{Theorem}[section]
\newtheorem{lemma}{Lemma}[section]
\newtheorem{rem}{Remark}[section]
\newtheorem{prop}{Proposition}[section]

\newtheorem{defin}{Definition}[section]

\newcommand{\Z}{{\mathbb{Z}}}
\newcommand{\C}{{\mathbb{C}}}
\newcommand{\R}{{\mathbb{R}}}
\newcommand{\Q}{{\mathbb{Q}}}

\newcommand{\pp}{{\mathbb{P}}}
\newcommand{\G}{{\mathcal{G}}}

\newcommand{\A}{{\mathcal{A}}}

\newcommand{\emme}{{\mathcal M}}
\newcommand{\elle}{{\mathcal L}}

\title{Projective Wonderful Models for Toric Arrangements}
\author{Corrado De Concini, Giovanni Gaiffi
}
\date{\today}

\begin{document}
\maketitle
\begin{abstract}
%\todo{DA RISCRIVERE alla fine}
In this paper we illustrate an algorithmic procedure which allows to build projective  wonderful models for  the complement of 
a toric  arrangement in a \(n\)-dimensional algebraic torus \(T\). The main step of the construction,  inspired by \cite {DCPII},  is  a combinatorial  algorithm that produces a toric variety by subdividing  in a suitable way a given smooth fan.

\end{abstract}
\section{Introduction}

Let us consider  a $n$-dimensional algebraic torus $T$ over the complex numbers. Let $X^*(T)$ denote its character group. This   is a lattice of rank $n$ and choosing a basis of $X^*(T)$  we get an isomorphism   \(T\simeq (\C^*)^n\). 

If we take a split direct summand $\Gamma\subset X^*(T)$ and a homomorphism $\phi:\Gamma \to \mathbb C^*$,
 we can consider  the subvariety, which will be called a layer, in $T$
$$\mathcal K_{\Gamma,\phi}=\{t\in T|\, \chi(t)=\phi(\chi),\, \forall \chi\in \Gamma\}.$$
Notice that a layer is a coset for the subtorus $H=\cap_{\chi\in\Gamma}Ker(\chi)$. So it is itself isomorphic to a torus and in particular it is smooth and irreducible.

A  toric arrangement \(\A\) is given by   finite set of  layers \(\A=\{\mathcal K_{1},...,\mathcal K_{m}\}\) in $T$.  We will say that a toric  arrangement \(\A\) is {\em divisorial}  if for every \(i=1,...,m\) the layer \(\mathcal K_i\)  has codimension 1.

In this paper we show how to construct a {\em projective wonderful model} for  the complement \(\emme(\A)=T-\bigcup_i \mathcal K_i\), i.e. a smooth projective  variety \(\mathcal W (\A)\) containing \(\emme(\A)\) as an open set and such that \(\mathcal W (\A) - \emme(\A)\) is  a divisor with normal crossings and smooth irreducible components.

Let us first  shortly recall  the state of the art  about toric arrangements. 
The study of toric arrangements  started  in \cite{LOO}. In the case of a divisorial arrangement,   it received a new impulse from several recent works for instance, in \cite{deconciniprocesivergne} and \cite{DCP3}  the role of toric arrangements as a link between partition functions and box splines is pointed out; interesting  enumerative and combinatorial aspects have been investigated via the Tutte polynomial and arithmetics matroids in \cite{mocitoricroot}, \cite{mocituttetoric}, \cite{dadderiomocitoric}.  As for the topology of the complement of a divisorial toric arrangement,  the generators of the cohomology modules over \(\C\) where exhibited  in \cite{DCP4} via local no broken circuits sets,  and in the same paper the cohomology ring structure was determined in the case of totally unimodular arrangements.  
A presentation of the fundamental group of the complement of a divisorial complexified  toric arrangement was provided in \cite{dantoniodelucchi1}, and in \cite{dantoniodelucchi2} d'Antonio and Delucchi proved that   \(\emme(\A)\) has the homotopy type of a minimal CW-complex and that its  integer cohomology is torsion free. 

Moreover, in  \cite{callegarodelucchi} Callegaro and Delucchi  computed the cohomology ring  with integer coefficients of \(\emme(\A)\) and started to investigate its dependency from the combinatorial data of the arrangement. 

The problem of finding a  wonderful model for \(\emme(\A)\) was first studied by Moci in \cite{mociwonderful}, where a construction of a non projective  model was described.

To explain the interest in the  construction of a projective wonderful model, we  briefly recall some results  in the case of subspace arrangements.

In \cite{DCP2}, \cite{DCP1},   De Concini and Procesi constructed  {\em wonderful  models}  for the   complement of a subspace arrangement in a  vector space (providing both a projective and a non projective version of the construction), as  an  approach to Drinfeld construction of special
solutions for Khniznik-Zamolodchikov equation (see \cite{drinfeld}).  
 Then real and complex  De Concini-Procesi models of subspace arrangements 
 were investigated from  several  points of view:  their cohomology  was studied for instance in \cite{yuzBasi}, \cite{etihenkamrai}, \cite{rains};   some  relevant combinatorial properties and their relation with discrete geometry were pointed out   in \cite{feichtner}, \cite{GaiffiServenti2}, \cite{gaiffipermutonestoedra}, \cite{callegarogaiffi3};   the case of complex reflection groups was dealt with  in \cite{hendersonwreath} from the representation theoretic point of view and in \cite{callegarogaiffilochak} from the homotopical point of view; relations with toric and tropical geometry were enlightened for  instance in  \cite{feichtnersturmfels} and \cite{denham}.
 
Furthermore, we recall that  in \cite{DCP1} it was shown, using the cohomology description of the projective  wonderful models  to give an explicit presentation of  a Morgan algebra, that the mixed Hodge numbers and the rational homotopy type of the complement of a complex  subspace arrangement  depend only on the intersection lattice (viewed as a ranked poset).  

By  analogy with the linear case, one of the reasons for the interest in the construction of a projective wonderful model for \(\emme(\A)\)  is  the computation of  the  Morgan algebra associated to the model and the investigation of its role in the study of  the dependency of the cohomology ring of \(\emme(\A)\) from the initial combinatorial data.
We leave this as a future direction of research.

Let us now describe more in detail the content of the present  paper.  

In Section \ref{sec:wonderfulmodels} we are going to briefly recall   the construction of wonderful models of varieties  with a conical stratification  in the sense of MacPherson-Procesi  \cite{procesimacpherson}, or, in other words containing an {\em arrangement of subvarieties} in the sense of Li  \cite{li}.

In  Section \ref{sec:layer}, given a smooth fan $\Delta$ in the  vector space   $hom_\Z(X^*(T),\R)= hom_\Z(X^*(T),\Z)\otimes_{\Z} \R$ and a layer $\mathcal K_{\Gamma,\phi}$, we are going to give a simple combinatorial condition which allows us to explicitly describe  the closure $\overline {\mathcal K}_{\Gamma,\phi}$ in the toric variety $K_{\Delta}$ corresponding to $\Delta$ and the intersection of $\overline {\mathcal K}_{\Gamma,\phi}$ with every $T$-orbit closure in $K_{\Delta}$.

Then, given a toric arrangement \(\A\) in \(T\) we will  construct a projective wonderful model for the complement \(\emme(\A)\) according to  the following strategy:
\begin{enumerate}

\item[1)] As a first step, we  construct (see Sections \ref{secalgoritmooriginale} and \ref{sec:strategia}) a smooth   projective $T$-variety \(K_{\Delta(\A)}\) (where \(\Delta(\A)\) denotes its fan). 
%We denote by \({\overline K_{\chi_i,b_i}}\)  the closures of the subtori \(K_{\chi_i,b_i}\)  in \(K_{\Delta(\A)}\). 

The crucial property of the toric variety \(K_{\Delta(\A)}\) is the following one.  Let us denote by \({\mathcal Q}\) the set whose elements are  the closures \({\overline {\mathcal K}_{i}}\)  of our layers and the irreducible components \(D_\alpha\) of  \(K_{\Delta(\A)}- T\).  Then   the family \(\elle\) of all the connected components of   intersections of   elements of \({\mathcal Q}\) gives an {\em arrangement of subvarieties}  in the sense of  Li's paper \cite{li}, 
as we will show  by a precise description   of the closure in  \(K_{\Delta(\A)}\)  of every subvariety in \(\elle\) .

\item[2)]  As a consequence of point \(1)\), for every choice of  a {\em building set} associated to the arrangement of subvarieties in \(K_{\Delta(\A)}\)  one can   obtain a projective wonderful model of \(\emme(\A)\).

\end{enumerate}

The construction of the toric variety \(K_{\Delta(\A)}\) is the result of a combinatorial algorithm    on fans that starts from the fan of \((\pp^1)^n\). This algorithm, which is a variant of an algorithm    introduced in \cite{DCPII} for a different purpose, is described  in Section \ref{secalgoritmooriginale}  and illustrated  by some examples  in Section  \ref{sec:example}. 

In Section \ref{sec:jacobian} we prove that the family of subvarieties \(\elle\) in \(K_{\Delta(\A)}\)  is an arrangement of subvarieties. 
The last section (Section \ref{sec:examples})  is devoted to some  remarks on our  construction.
First we show that, although our construction is not canonical (it depends for instance from the initial identification of the fan of \((\pp^1)^n\)),  in some cases there is also a more canonical way to obtain a toric variety \(K_{\Delta(\A)}\) with the requested properties. This happens for instance for  divisorial toric arrangements \(\A\)  associated to  root systems or   to a directed graph.

Finally we show that  if \(\mathcal W(\A)\) is a projective wonderful model obtained by our construction, then its integer cohomology  is even and torsion free and the cohomology ring is isomorphic to the Chow ring (i.e. \(\mathcal W(\A)\) has property \((S)\)  according to the definition in \cite{DCLP} 1.7). This follows from the description of the strata in   Section \ref{sec:layer} and from the fact that the construction of wonderful models in \cite{procesimacpherson}, \cite{li} can be seen as the result of  a prescribed sequence of blowups.

\section{Wonderful models of stratified varieties}
\label{sec:wonderfulmodels}

In the literature one can find several  general constructions that, starting from a `good' stratified variety, produce models by blowing up a suitable subset of  strata.  
For instance,  as we mentioned in the Introduction, the case of the stratification induced in a vector space by a subspace arrangement is discussed in    \cite{DCP2}, \cite{DCP1}.

The papers  of MacPherson and Procesi  \cite{procesimacpherson} and   Li \cite{li} extend the construction of wonderful models from the linear case to  the more general setting of a variety stratified by a set of subvarieties.

 In Li's paper one can also find a  comparison among several  constructions of models, including the ones   by  Fulton-Machperson (\cite{fultonmacpherson}),  Ulyanov (\cite{Ulyanov}) and  Hu (\cite{Hu}). Denham's paper \cite{denham} provides a further interesting survey including tropical compactifications.

We recall here some definitions and results  from \cite{procesimacpherson} and   \cite{li} adopting the language  and the notation of Li's paper.

\begin{defin}
\label{def:simple} A simple {\em arrangement of subvarieties} of a nonsingular variety \(Y\) is a finite set \(\Lambda = \{\Lambda_i\}\) of nonsingular closed connected subvarieties \(\Lambda_i\), properly contained in \(Y\), that satisfy the following conditions:\\
(i)  \(\Lambda_i\) and \(\Lambda_j\)  intersect  {\em cleanly},  i.e. their intersection  is nonsingular and for every \(y\in  \Lambda_i\cap \Lambda_j\) we have 
\[T_{\Lambda_i\cap \Lambda_j,y} = T_{\Lambda_i,y}\cap T_{\Lambda_j,y}\]
(ii) \(\Lambda_i \cap \Lambda_j\)  either   belongs to  \(\Lambda\)  or is empty.

\end{defin}

\begin{defin}
Let \(\Lambda\) be a simple  arrangement of subvarieties of  \(Y\).  A subset \(\G \subseteq \Lambda\) is called a {\em building set} of \(\Lambda\)  if for every \(\Lambda_i \in \Lambda - \G\)  the minimal elements in \(\{G  \in \G \: : \: G \supseteq \Lambda_i\}\)  intersect transversally and their intersection is \(\Lambda_i\). These minimal elements are called the {\em \(\G\)-factors} of \(\Lambda_i\).

\end{defin}

\begin{defin}
Let \(\G\) be  a building set of a simple arrangement \(\Lambda\).
A subset \({\mathcal T}\subseteq \G\) is called \(\G\)-nested if it satisfies the following condition: if \(A_1, ...,A_k\) are the minimal elements of  \({\mathcal T}\) (with \(k>1\)), then they are the \(\G\)-factors of an element in \(\Lambda\). 
Furthermore, for any \(i\), the set \(\{A\in {\mathcal T} \: | \: A \supsetneq A_i\}\) is also nested as defined by induction.
\end{defin}

We remark that in Section 5.4 of \cite{li} some  even  more general definitions are provided, to include the case when the intersection of two strata is a disjoint union of strata. Since this  will be useful  for   our  toric stratifications, we recall these definitions  in detail.
\begin{defin}
\label{def:nonsimple} An {\em arrangement of subvarieties} of a nonsingular variety \(Y\) is a finite set \(\Lambda = \{\Lambda_i\}\) of nonsingular closed connected subvarieties \(\Lambda_i\), properly contained in \(Y\), that satisfy the following conditions:\\
(i)  \(\Lambda_i\) and \(\Lambda_j\)  intersect  cleanly;\\
(ii) \(\Lambda_i \cap \Lambda_j\)  either   is equal to the disjoint union of some    \(\Lambda_k\)  or is empty.

\end{defin}

\begin{defin}
Let \(\Lambda\) be an arrangement of subvarieties of  \(Y\).  A subset \(\G \subseteq \Lambda\) is called a building set of \(\Lambda\)  if there is an open cover \(\{U_i\}\) of \(Y\) such that:\\
a) the restriction of the arrangement \(\Lambda\) to \(U_i\) is simple for every \(i\);\\
b) \( \G_{|U_i}\)  is a building set of \(\Lambda_{|U_i}\).

\end{defin}

\begin{defin}
Let \(\G\) be  a building set of an arrangement \(\Lambda\).
A subset \({\mathcal T}\subseteq \G\) is called \(\G\)-nested if there is an open cover \(\{U_i\}\) of \(Y\) such that
\({\mathcal T}_{|U_i}\) is  \(\G_{|U_i}\)-nested for every \(i\).
\end{defin}

Then, if one has an arrangement  \(\Lambda\) of a nonsingular variety \(Y\) and a  building set \(\G\), one can construct a wonderful model \(Y_{\G}\)   by  considering  (by  analogy with \cite{DCP1})  the closure of the image of the  locally closed embedding
\[\left( Y-\bigcup_{\Lambda_i\in \Lambda}\Lambda_i \right ) \rightarrow \prod_{G\in \G}Bl_GY\] where \(Bl_GY\) is the blowup of \(Y\) along \(G\).

It turns out that: 

\begin{teo}
\label{teo:listabuilding}
The variety \(Y_{\G}\) is nonsingular. If one arranges the elements \(G_1,G_2,...,G_N\)  of \(\G\) in such a way that for every \(1\leq i \leq N\) the set \(\{G_1,G_2,\ldots , G_i\}\) is building, then \(Y_{\G}\) is isomorphic to the variety  \[Bl_{{\widetilde G_N}}Bl_{{\widetilde G_{N-1}}}\
\cdots Bl_{{\widetilde G_2}}Bl_{G_1}Y\]
where \({\widetilde G_i}\) denotes the dominant transform of $G_i$ in $Bl_{{\widetilde G_{i-1}}}
\cdots Bl_{{\widetilde G_2}}Bl_{G_1}Y$.
\end{teo}

%We will not recall here   the theorems that describe the boundary of these more general wonderful models, but  

\begin{rem}
\label{rem:ordinescoppiamenti}
As remarked by MacPherson-Procesi  in \cite[Section 2.4]{procesimacpherson} it is
always possible to choose a linear ordering on the set $\G$ such that every initial segment is building. We can do this by ordering $\G$ in such a way that we always blow up first the strata of smaller  dimension.
\end{rem}

Another theorem (see \cite{procesimacpherson}, \cite{li}) describes the boundary of \(Y_\G\) in terms of \(\G\)-nested sets:
\begin{teo}
For every \(G\in \G\) there is a nonsingular divisor \(D_G\) in \(Y_\G\);  the union of these divisors is 
the complement in \(Y_\G\) to \(Y-\bigcup_{\Lambda_i\in \Lambda}\Lambda_i \). An intersection of divisors \(D_{T_1}\cap \cdots \cap D_{T_k}\) is nonempty if and only if \(\{T_1,...,T_k\}\) is \(\G\)-nested. If the intersection  is nonempty it is transversal.

\end{teo}

\section {The closure of a layer in  a toric variety}
\label{sec:layer}

Let us start with a very simple fact. Let $V$ be a real vector space. $B=\{e^1,\ldots e^h\}$ a set  of linearly independent vectors in $V$. We denote by $C(B)$ the cone of nonnegative linear combinations of the $e^i$'s.

Given  a subspace $U\subset V^*$,  we say that $U$ has property $(E)$ with respect to $C(B)$ if there is a basis $u_1,\ldots ,u_r$ of $U$ such that $\langle u_i,e^j\rangle\geq 0$ for all $i=1,\dots ,r$, $j=1,\ldots ,h$. We set $U^{\perp}=\{w\in V|\, \langle u, w\rangle=0,\, \forall\, u\in U\}$. 
It is now easy to show that 
\begin{lemma}\label{iconetti} Assume that $U$ has property $(E)$ with respect to $C(B)$. Then $$C(B)\cap U^\perp=C(B\cap U^\perp)\,  \text{(if}\  B\cap U^\perp=\emptyset,\,  C(B\cap U^\perp)=\{0\}).$$ \end{lemma}

Let us  take $V=hom_\Z(X^*(T),\R)=X_*(T)\otimes_\Z\R$, with $X_*(T):=hom_\Z(X^*(T),\Z)$ the lattice of one parameter subgroups in $T$. Then we have a natural identification of $T$ with $V/X_*(T)$ and we may consider a $\chi\in X^*(T)$ as a  linear function on $V$. From now on the corresponding character  $e^{2\pi i\chi}$ will be usually denoted by $x_\chi$.
Recall the definition of a layer:

\begin{defin} Given  a split direct summand $\Gamma\subset X^*(T)$ and a homomorphism $\phi:\Gamma \to \mathbb C^*$, the subvariety
$$\mathcal K_{\Gamma,\phi}=\{t\in T|\, x_\chi(t)=\phi(\chi),\, \forall \chi\in \Gamma\}$$
will be called a layer. \end{defin}
We have already remarked that $\mathcal K_{\Gamma,\phi}$ is  a coset with respect to the subtorus 
$H=\cap_{\chi\in \Gamma}Ker(x_\chi)$.
  and consider the subspace $V_H=\{v\in V|\, \langle\chi,v\rangle=0,\, \forall \chi\in \Gamma\}$. Notice that since $X^*(H)=X^*(T)/\Gamma$, $V_H$ is naturally isomorphic to $hom_\Z(X^*(H),\R)= X_*(H)\otimes_{\Z} \R$.

Assume now we are given a smooth  fan  in $V$, that is a collection $\Delta$ of simplicial cones in $V$ such that \begin{enumerate}
\item Each cone $C\in \Delta $ is the cone $C(e^1,\ldots e^r)$  of  non negative linear combinations of  linearly independent vectors $e^1,\ldots e^r$ in the lattice $X_*(T)$ spanning a split direct summand.
\item If $C\in\Delta$ every face of $C$ is also in $\Delta$.
\item If $C,C'\in\Delta$, $C\cap C'$ is a face of \(C\) and of \(C'\). \end{enumerate}

\begin{defin} The layer  $\mathcal K_{\Gamma,\phi}$ has property $(E)$ with respect to the fan $\Delta$ if the subspace $\Gamma\otimes_\Z\R\subset X_*(T)\otimes_\Z\R $ has property $(E)$ with respect to every cone $C\in \Delta$.\end{defin}
\begin{rem} Notice that the condition of having property $(E)$ with respect to $\Delta$ depends only on $\Gamma$, in fact only on the vector space $\Gamma\otimes_{\Z}\R$,  and not on the homomorphism $\phi$.\end{rem}

\begin{lemma}\label{aritmetico} Assume that the layer  $\mathcal K_{\Gamma,\phi}$ has property $(E)$ with respect to  the cone $C=C(e^1,\ldots ,e^h)$, $e^i\in X_*(T)$ for each $i=1,\ldots h$. Then there is an integral  basis of $\Gamma$, $\chi_1,\ldots \chi_r$, such that 
$\langle\chi_i,e^j\rangle\geq 0$ for all $i=1,\dots ,r$, $j=1,\ldots ,h$.\end{lemma}
\begin{proof}
First of all we can assume that $r\leq h$. Indeed, otherwise consider the sublattice  $\Gamma'=\Gamma\cap \langle e^1,\ldots ,e^h\rangle^\perp$ of elements in $\Gamma$ orthogonal to the $e^{j}$'s. $\Gamma'$  is a direct summand in $\Gamma$, so choosing a complement $\Gamma''$ we clearly have that $rk(\Gamma'')\leq h$ and that 
the space  $\Gamma''\otimes_\Z\R$ has property $(E)$ with respect to  $C$.

It now clearly suffices to prove our statement for $\Gamma''$. So let us assume $r\leq h$ and furthermore that  $\Gamma\cap  \langle e^1,\ldots ,e^h\rangle^\perp=\{0\}.$

Now under our assumptions, there is a basis of $\Gamma\otimes_{\Z}\R$, $\psi_1,\ldots \psi_r$ with $\langle\psi_j,e^i\rangle\geq 0$ for all $i,j$. 

Furthermore for every  $i=1,\ldots ,h$ there is a $j(i)$ such that  $\langle\psi_{j(i)},e^i\rangle>0$.
Setting $\psi=\sum_i\psi_{j(i)}$, we see that $\psi$ is strictly positive on $C$.

Since $\Q$ is dense in $\R$ we immediately deduce that we can choose the $\psi_j$'s in 
$\Gamma\otimes_{\Z}\Q$ and, clearing denominators, even in $\Gamma$. So $\psi_1,\ldots \psi_r$ span a sublattice of finite index in $\Gamma$.  Also in this situation, the vector $\psi\in\Gamma$ and,  eventually dividing by a positive integer, we find a primitive vector $\chi\in \Gamma$ which is strictly positive on $C$.

Let us complete $\chi$ to an integer basis $\gamma_1=\chi,\gamma_2, \ldots ,\gamma_r$ of $\Gamma$. Then there  is a positive integer $N$ such that $\chi_j:=\gamma_j+N\gamma_1$ for $j=2,\ldots r$ is non negative on $C$. We deduce that the integer basis $\chi_1=\chi,\chi_2, \ldots ,\chi_r$ of $\Gamma$ satisfies all the required properties.\end{proof}

Let us denote by $K_\Delta$ the smooth $T$-variety associated to the fan $\Delta$ and by $\overline {\mathcal K}_{\Gamma,\phi}$ the closure of the layer  $\mathcal K_{\Gamma,\phi}$ in $K_\Delta$.  Notice that $H$ clearly acts on $\overline {\mathcal K}_{\Gamma,\phi}$ with dense orbit $\mathcal K_{\Gamma,\phi}$. From Lemma \ref{iconetti} we deduce,
\begin{prop}\label{faneH} Assume that  $\mathcal K_{\Gamma,\phi}$ has property $(E)$ with respect to the fan $\Delta$. Then: 
\begin{enumerate}\item[1)] For every cone $C\in \Delta$, its relative interior is either entirely contained in $V_H$ or disjoint from $V_H$.\item[2)] The collection of cones $C\in \Delta$ which are contained in $V_H$ is a smooth fan $\Delta_H$.\end{enumerate}
\end{prop}
\begin{proof} Notice that $X_*(T)\cap V_H=X_*(H)$. From this  and Lemma \ref{iconetti} we deduce that for every cone $C\in\Delta$ the intersection $C\cap V_H$ is a face of $C$. If $C\cap V_H$ is a proper face of $C$ then the relative interior of $C$ is disjoint from $V_H$, otherwise  
 $C\subset V_H$. This gives 1).
 
 As for 2), notice that from 1) the collection  $\Delta_H$ of faces  $C\in \Delta$ which are contained in $V_H$ is a fan in $V_H$. To see that it is smooth, it suffices to remark that since $X_*(H)$ is a direct summand in $X_*(T)$, a sublattice of $X_*(H)$ is  direct summand of $X_*(T)$  if and only if it is  a direct summand in $X_*(H)$.\end{proof}
We know that there is a bijection between the fan $\Delta$ and on the one hand the set of $T$ stable affine open sets $K_\Delta$ , on the other hand  the set  of $T$ orbits in $K_\Delta$. To give these bijections,
let $C$  be a face of $\Delta$. Set $$D_C=\{\chi\in X^*(T)|\langle \chi, v\rangle\geq 0,\ \forall v\in C\},\ \ \ D^+_C=\{\chi\in D_C| \chi_{|C}\neq 0\}.$$ Then the affine open set $U_C\subset K_{\Delta}$ has coordinate ring $\mathbb C[U_C]=\sum_{\chi\in D_C} \C x_\chi \subset \C [T]$ and the ideal of the unique relatively closed orbit $\mathcal O_C$ in $U_C$ is given by 
$I_C=\sum_{\chi\in D^+_C} \C x_\chi $.

 The geometric counterpart of Proposition \ref{faneH} is
\begin{teo} \label{cllay}Assume that  $\mathcal K_{\Gamma,\phi}$ has property $(E)$ with respect to the fan $\Delta$.Then \begin{enumerate}
\item[1)]  $\overline {\mathcal K}_{\Gamma,\phi}$ is a smooth $H$-variety whose fan is $\Delta_H$.
\item[2)] Let $\mathcal O$ be a $T$ orbit in $K_\Delta$ and let $C_{\mathcal O}\in \Delta$ be the corresponding cone. Then 
\begin{enumerate}
\item If $C_{\mathcal O}$ is not contained in $ V_H$, $C_{\mathcal O}\cap \overline {\mathcal K}_{\Gamma,\phi}=\emptyset$.
\item If $C_{\mathcal O}\subset V_H$, $C_{\mathcal O}\cap \overline {\mathcal K}_{\Gamma,\phi}$ is the $H$ orbit in $\overline {\mathcal K}_{\Gamma,\phi}$ corresponding to 
$C_{\mathcal O}\in \Delta_H$. 
\end{enumerate} \end{enumerate}\end{teo}

\begin{proof}  
1) Since the  affine $T$-stable open sets cover $K_\Delta$, 
to see that $\overline {\mathcal K}_{\Gamma,\phi}$ is smooth , it suffices to show that its intersection with every affine $T$-stable open set is smooth.

So fix a cone $C\in \Delta$ and let $U_C\subset K_\Delta$ be the corresponding open set.
If $C=C(e^1,\ldots e^s)$, then by assumption we can complete $e^1,\ldots e^s$ to an integral basis $e^1,\ldots e^n$ of $X_*(T)$ and by taking the dual basis $(\chi_1,\ldots \chi_n)$ of $X^*(T)$
we obtain an identification of $R$ with 
$\mathbb C[x_1,\ldots x_s,x_{s+1}^{\pm 1},\ldots x_n^{\pm 1}]$, where we set $x_i=x_{\chi_i}$, $i=1,\ldots ,n$, and hence of $U_C$ with $\mathbb A^{s}\times \mathcal O$.
Now take a basis $\mu_1,\ldots \mu_r$ of $\Gamma$. Since property $(E)$ holds we can assume  by Lemma \ref{aritmetico}, that $\langle \mu_j,e^i\rangle\geq 0$ for all $j=1,\ldots r$, $i=1,\ldots ,s.$ 
%{\textcolor {red} {Qui bisogna vedere che se la $(E)$ e' vera la base si puo' prendere intera.}}

Each $\mu_j\in R$ 
so that, setting $b_j=\phi(\mu_j)$, we get that  the ideal of $\overline {\mathcal K}_{\Gamma,\phi}\cap U_C$ is generated by the polynomials $p_1(x_1,\ldots ,x_n),\ldots ,p_r(x_1,\ldots ,x_n)$ with 
$$ p_j(x_1,\ldots ,x_n)=x_1^{m_{1,j}}\cdots x_n^{m_{n,j}}-b_j,\ \ \ \ j=1,\ldots ,r.$$
Remark that by the linear independence of the $\mu_j$'s, the  matrix  $A=(m_{i,j})$ has maximal rank $r$ and non negative integer entries. So, there is a sequence $1\leq i_1<\cdots <i_r\leq n$ such that the determinant of the $r\times r$ matrix $C=(m_{i_\ell,t})$ is non zero. 
Set $M_i=\sum_{j=1}^rm_{i,j}.$
%In particular $M_{j_\ell}=\sum_{i=1}^r|m_{j_\ell,i}| >0$, for each $\ell=1,\ldots ,r$.
A simple computation shows that
$$\det (\partial p_j/\partial x_{i_\ell})=\det (C)\prod_{\ell=1}^rx_{i_\ell}^{M_{i_\ell}-1}\prod_{i\neq i_\ell}x_{i}^{M_i}$$
Since the polynomial  $\prod_\ell x_{i_\ell}^{M_{i_\ell}-1}\prod_{i\neq i_\ell}x_{i}^{M_i}$ does not vanish on  $U_C\cap \overline {\mathcal K}_{\Gamma,\phi}$ , and $\det (C)\neq 0$,we deduce that $U_C\cap \overline {\mathcal K}_{\Gamma,\phi}$ is smooth as desired.

2) We keep the notations  introduced above. First assume that $C$ is not contained in $V_H$ so that there is $\chi\in \Gamma$ and $v\in C$ such that $\langle\chi ,v\rangle\neq 0$. It follows that there is at least one pair $(i,j)$ with $i=1,\ldots s$ and $j=1,\ldots r$ such that $m_{i,j}>0$.
Since $x_i\in I_C$ and the $b_j$'s are non zero,   we deduce that the ideal $(I_C,p_j)$ is the unit ideal proving that $\mathcal O\cap \overline {\mathcal K}_{\Gamma,\phi}=\emptyset$.

Assume now that $C=C(e^1,\ldots e^s)\subset V_H$. We complete $e^1,\ldots ,e^s$ to a basis of $X_*(T)$ by first completing  $e^1,\ldots e^s$ to a basis 
$e^1,\ldots ,e^{n-r}$ of $X_*(H)$ and then adding $r$ vectors $e^{n-r+1},\ldots e^n$ to get a basis of $X_*(T)$. Let us now consider the basis $\chi_1,\ldots \chi_n$ of $X_*(T)$ dual to the  basis chosen above. We know that the coordinate ring of $U_C$ is given by  $\mathbb C[x_1,\ldots x_s,x_{s+1}^{\pm 1},\ldots x_n^{\pm 1}]$. 

Clearly $\chi_{n-r+1},\ldots \chi_n$ is a basis of $\Gamma$ and setting $a_i=\phi(\chi_{n-r+i})$, $i=1,\ldots ,n$ we get that  the ideal $J$ of $\overline {\mathcal K}_{\Gamma,\phi}\cap U_C$ is generated by the polynomials $$ x_{n-r+i}-a_i,\ \ \ \ i=1,\ldots ,r.$$
It follows immediately that we have a $H$ equivariant isomorphism $$\mathbb C[x_1,\ldots x_s,x_{s+1}^{\pm 1},\ldots x_n^{\pm 1}]/J\simeq\mathbb C[x_1,\ldots x_s,x_{s+1}^{\pm 1},\ldots x_{n-r}^{\pm 1}]$$
Thus  a $H$ invariant affine open set in the $H$ variety $\overline {\mathcal K}_{\Gamma,\phi}$, $\overline {\mathcal K}_{\Gamma,\phi}\cap U_C$ corresponds to the cone $C$. Furthermore the unique $H$ closed orbit in  $\overline {\mathcal K}_{\Gamma,\phi}\cap U_C$ coincides with  $\overline {\mathcal K}_{\Gamma,\phi}\cap \mathcal O_C$. 

To finish, let us remark that, if we take any $H$ orbit $\mathcal P$ in $\overline {\mathcal K}_{\Gamma,\phi}$, then if we choose $p\in \mathcal P$ there is a $T$ orbit $\mathcal O$ in $K_\Delta$ such $p\in \mathcal O$. From the above analysis it follows that the cone $C_{\mathcal O}\subset V_H$ and hence $\mathcal P=\mathcal O\cap \overline {\mathcal K}_{\Gamma,\phi}$  proving   our claims.

\end{proof}

\begin{rem}
\label{rem:strati}
 \begin{enumerate}\item Notice that by Theorem \ref{cllay}, if   in addition the fan $\Delta$ is complete  and $\mathcal K_{\Gamma,\phi}$ has property $(E)$ with respect to $\Delta$, then   the space $V_H$ is the  union of cones  of  $\Delta$ that  it contains.

\item Under the same assumptions we clearly also have that for any $T$ orbit closure $\overline O$ in $K_\Delta$, the intersection $\overline {\mathcal K}_{\Gamma,\phi}\cap \overline {\mathcal O}$ is either empty or consists of  a $H$ -orbit closure in $\overline {\mathcal K}_{\Gamma,\phi}$. In this case it is clean.\end{enumerate}\end{rem}

\section{A combinatorial algorithm}
\label{secalgoritmooriginale}
In this section we describe a combinatorial algorithm such that,   starting from  a  finite set of vectors  $\Xi$ in a lattice $L$ and a smooth fan  $\Delta$ in $V=hom_\Z(L,\R)$,  produces a new fan $\overline \Delta$ with the same support as $\Delta$ (a proper subdivision of $\Delta$) with the property that, for each cone $C\in \overline \Delta$ and each $\chi \in \Xi$  we either have $\langle \chi, C\rangle\geq 0$ or
$\langle \chi, C\rangle\leq 0$. In other words  the line $\mathbb R\chi$ has property $(E)$ with respect to $C$. In view of this we shall say that $\mathbb \chi$ has property $(E)$ with respect to $C$. Notice that it suffices to check property $(E)$ on each two dimensional face $C=C(e^1,e^2)$ where it is equivalent to $\langle \chi, e^1\rangle\langle \chi, e^2\rangle\geq 0$.

A closely related algorithm already appears, for different, although related, purposes in  \cite{DCPII}. Here we give an alternative simplified version which we believe better explains the role of two dimensional faces. 

Let us start with a single vector $\chi$. If all cones in $\Delta$ are one dimensional there is clearly nothing to prove. So let us assume that $\Delta$ contains at least a cone of dimension $2$.

The algorithm consists of repeated applications of the following move.
\begin{itemize}
\item Start with the fan $\Delta$. If  $\chi$ has property $(E)$ with respect to  each  two dimensional cone in $\Delta$,  then $\Delta$ already has the required properties and we stop. Otherwise,
\item  Suitably choose a two dimensional face $C=C(e^1,e^2)$ of $\Delta$ with the property that 
$\langle \chi, e^1\rangle\langle \chi, e^2\rangle< 0$.
\item Define the new fan $_C \Delta$ which is obtained from $\Delta$ by substituting each  cone \(C(e^1,e^2,w^1,...,w^t)\) containing $C$ with the two cones \(C(e^1,e^1+e^2,w^1,...,w^t)\) and \(C(e^1+e^2,e^2,w^1,...,w^t)\).

\end{itemize}

The following Proposition is clear and we leave it to the reader.
\begin{prop}\label{lublup} \begin{enumerate}\item[1)] The fan $_C\Delta$ is smooth.
\item[2)] $_C \Delta$ is a proper (and in fact projective) subdivision of $\Delta$.
\item[3)] If $L=X^*(T)$ and $K_\Delta$ and $K_{_C\Delta}$ are the $T$-varieties corresponding to $\Delta$ and $_C\Delta$, $K_{_C\Delta}$ is obtained from $K_{\Delta}$  blowing up the closure of the  orbit of codimension two in $K_{\Delta}$ associated to $C$.\end{enumerate}
 \end{prop}

In view of Proposition \ref{lublup} what we have to show is that we can judiciously make a sequence of the above moves in such a way that at the end we obtain a fan with the required properties.

We denote by $\Delta^{(2)}$ the set of two dimensional cones  in $\Delta$. 
\begin{lemma}\label{newcone} A cone in $_C\Delta^{(2)}$ is either   a cone in $\Delta^{(2)}\setminus \{C\}\cup \{C(e^1,e^1+e^2),C(e^2,e^1+e^2)\}$ or it is of the form $C(e^1+e^2,u)$ with $C(e^1,e^2,u)\in\Delta$\end{lemma}

We set $\Delta^{(2)}_N\subset \Delta^{(2)}$ equal to the set of cones with respect to which $\chi$ does not have property $(E)$.

Whenever $\Delta^{(2)}_N\neq \emptyset$, we define 
$$P_\Delta:\Delta^{(2)}_N\to \mathbb N\times \{0,1\}$$
by setting for $\sigma=C(\eta^1,\eta^2)$, $P_\Delta(\sigma)=(M_\sigma,\varepsilon_\sigma)$ with $M_\sigma=max_{s=1,2}| \langle \chi,  \eta^s\rangle |\), and $\varepsilon_\sigma= 1$ if 
$| \langle \chi,  \eta^1\rangle |=| \langle \chi,  \eta^2\rangle |=M_\sigma$, $\varepsilon_\sigma= 0$ otherwise.

Let us now order the set $\mathbb N\times \{0,1\}$ lexicographically. We have
\begin{lemma} \label{lalgo}Assume $\Delta^{(2)}_N\neq \emptyset$ and choose $\sigma=C(e^1,e^2)\in \Delta^{(2)}_N$ in such a way that $P_\Delta(\sigma)=(M_\sigma,\varepsilon_\sigma)$ is maximum.

Then \begin{enumerate}

\item[1)] if $\varepsilon_\sigma=1$,  then $ _\sigma\Delta^{(2)}_N=\Delta^{(2)}_N\setminus \sigma$
\item[2)] If $\varepsilon_\sigma=0$, then the maximum value of $P_{_\sigma\Delta}$ is less than or equal than $(M_\sigma,\varepsilon_\sigma)$. Furthermore $|P_{_\sigma\Delta}^{-1}((M_\sigma,\varepsilon_\sigma))|<|P_{\Delta}^{-1}((M_\sigma,\varepsilon_\sigma))|$.
\end{enumerate}
\end{lemma}
\begin{proof} By eventually exchanging $e^1$ and $e^2$, and taking $-\chi$ instead of  $\chi$, we can always assume that $M_\sigma=\langle\chi,e^1\rangle>0> \langle\chi,e^2\rangle\geq -M_\sigma$.

By Lemma \ref{newcone}, we need to analyse the cones  $C(e^1,e^1+e^2),\ C(e^2,e^1+e^2)$ and $C(e^1+e^2,u)$ with $C(e^1,e^2,u)\in \Delta.$

1. Suppose $\varepsilon_\sigma=1$, then $\langle\chi,e^2\rangle= -M_\sigma$, hence $\langle\chi,e^1+e^2\rangle=0$ so that all these cones do not lie in $ _\sigma\Delta^{(2)}_N$. It follows that $ _\sigma\Delta^{(2)}_N= \Delta^{(2)}_N\setminus \sigma$  hence our claim.

2. If $\varepsilon_\sigma=0$, necessarily for any $u$ such that $C(e^1,u)\in \Delta^{(2)}_N$, $0>\langle\chi,u\rangle> -M_\sigma$. In particular $M_\sigma>\langle\chi,e^1+e^2\rangle>0$. 

We have
\begin{enumerate} \item $C(e^1,e^1+e^2)\notin \, _\sigma\Delta^{(2)}_N$.
\item $M_{C(e^2,e^1+e^2)}=max(-\langle\chi,e^2\rangle, \langle\chi,e^1+e^2\rangle)<M_\sigma$.
\item Assume $C(e^1,e^2,u)\in \Delta$. Then
\begin{enumerate} \item [a)] If $\langle\chi,u\rangle\geq 0$, $C(e^1+e^2,u)\notin\, _\sigma\Delta^{(2)}_N$.\item [b)]If $\langle\chi,u\rangle<0 $,  $M_{C(e^1+e^2,u)}=max(\langle\chi,e^1+e^2\rangle, -\langle\chi,u\rangle, )<M_\sigma$.\end{enumerate}\end{enumerate}
We deduce that  $P_{_\sigma\Delta}$ takes values which are at most equal to $(M_\sigma,\epsilon_\sigma)$. Furthermore
if $\tau\in  {_\sigma\Delta}^{(2)}_N$ is such that $P_{_\sigma\Delta}(\tau)=(M_\sigma,\epsilon_\sigma)$, necessarily $\tau\in \Delta^{(2)}_N\setminus\{\sigma \}$ and everything follows.
\end{proof}

Let us now denote by $\mathcal M_\Delta$   the family of fans which are obtained from $\Delta$ by
a repeated application of the following procedure: given a fan \(\mathcal R\), choose a two dimensional cone  \(\sigma\) in \(\mathcal R\) and create the new fan \(_\sigma{\mathcal R}\).

\begin{teo}[see also \cite{DCPII}]
\label{teoalgoritmooriginale}
Let $L$ be a lattice and $\Delta$  a smooth fan giving a partial rational decomposition of
$\hom(L,\mathbb R)$. Let $\Xi\subset L$ be a finite subset. Then there is \(\overline \Delta \in  \mathcal M_\Delta$ such that
\begin{enumerate}\item[1)] \(\overline \Delta \) is a smooth fan.
\item[2)] \(\overline \Delta \) is a projective subdivision of $\Delta$.
\item[3)] For every $\chi\in \Xi$, $\chi$ has property $(E)$ with respect to every cone in \(\overline \Delta \).
\end{enumerate}
\end{teo}
\begin{proof} 
The first two properties are obviously satisfied for every $\Theta\in \mathcal M_\Delta$. 
Let us show how to find  \(\overline \Delta \) satisfying the third.

We proceed by induction on the cardinality of $\Xi$. If $\Xi=\emptyset$ there is nothing to prove.
Let $\Xi=\{\chi\}$.
If $\Delta^{(2)}_N=\emptyset$ again  there is nothing to prove, \(\overline \Delta =\Delta\). 

Otherwise define for $\Theta\in\mathcal M_\Delta$, 
$$M_\Theta=\begin{cases}0\ \  \ \hskip2cm  \text{if}\ \Theta^{(2)}_N=\emptyset\\
max_{\sigma\in \Theta^{(2)}_N}M_\sigma \ \ \text{otherwise}\end{cases}$$
$$\varepsilon_\Theta=\begin{cases}0\ \ \   \ \hskip3cm  \text{if}\ \Theta^{(2)}_N=\emptyset\\
max_{\sigma\in \Theta^{(2)}_N, M_\sigma=M_\Theta}\varepsilon_\sigma \ \ \text{otherwise}\end{cases}$$
$$q_\Theta=\begin{cases}0\ \  \ \hskip2.4cm  \text{if}\ \Theta^{(2)}_N=\emptyset\\
|P_{\Theta}^{-1}((M_\Theta,\varepsilon_\Theta))| \ \ \text{otherwise}\end{cases}$$ 
Take $\Theta$ in such a way that the triple $(M_\Theta, \varepsilon_\Theta,q_\Theta)$ is lexicographically minimum. If $\Theta^{(2)}_N\neq\emptyset$, by Lemma \ref{lalgo} we can find a $\sigma\in \Theta^{(2)}_N$ such that
the triple $(M_{_\sigma\Theta}, \varepsilon_{_\sigma\Theta},q_{_\sigma\Theta})$ is smaller giving a contradiction. This settles the case $\Xi=\{\chi\}$

The general case now follows immediately by induction once we remark that if for a given $\chi$, $\Theta$ is such that $\Theta^{(2)}_N=\emptyset$, the for every $\sigma\in \Theta^{(2)}$ also $_\sigma\Theta^{(2)}_N=\emptyset$.
\end{proof}
\section{An example of how the algorithm works}
\label{sec:example}
Let us consider fan $\Delta$ in $\mathbb R^3$ consisting of the first quadrant together with its faces. Let $\Xi=
\{\chi_1=(3,0,-2), \chi_2=(2,1,-1)\}$.

As an example  of the strategy described in Section \ref{sec:strategia} we will show how to subdivide   \(\Delta\) getting another fan  $\overline \Delta$ with the property that both characters $\chi_1$ and $\chi_2$ have property $(E)$ with respect to every cone of  $\overline \Delta$,  that means that   with respect to each 3-dimensional cone of $\overline \Delta$, $\chi_1$ and $\chi_2$ are both expressed in the corresponding dual bases   with  all nonnegative or all nonpositive coordinates.

We   apply  our  algorithm   until  \(\chi_1\) has property $(E)$ with respect to each cone  and the coordinates we get at the end for $\chi_1$ are given by the set  \(X=\{(3,0,1), (-1,0,-2), (0,0,-1), (1,0,0)\}\) (see  the left hand side of Figure \ref{algoritmo1}). As far as $\chi_2$ is concerned, after these steps, the set of coordinates does not always satisfy property $(E)$. Indeed we get the coordinates  \( (0,1,-1)\) in one case (see  the right hand side of Figure \ref{algoritmo1}). 

\begin{figure}[h!]
\centering
\includegraphics[scale=0.2]{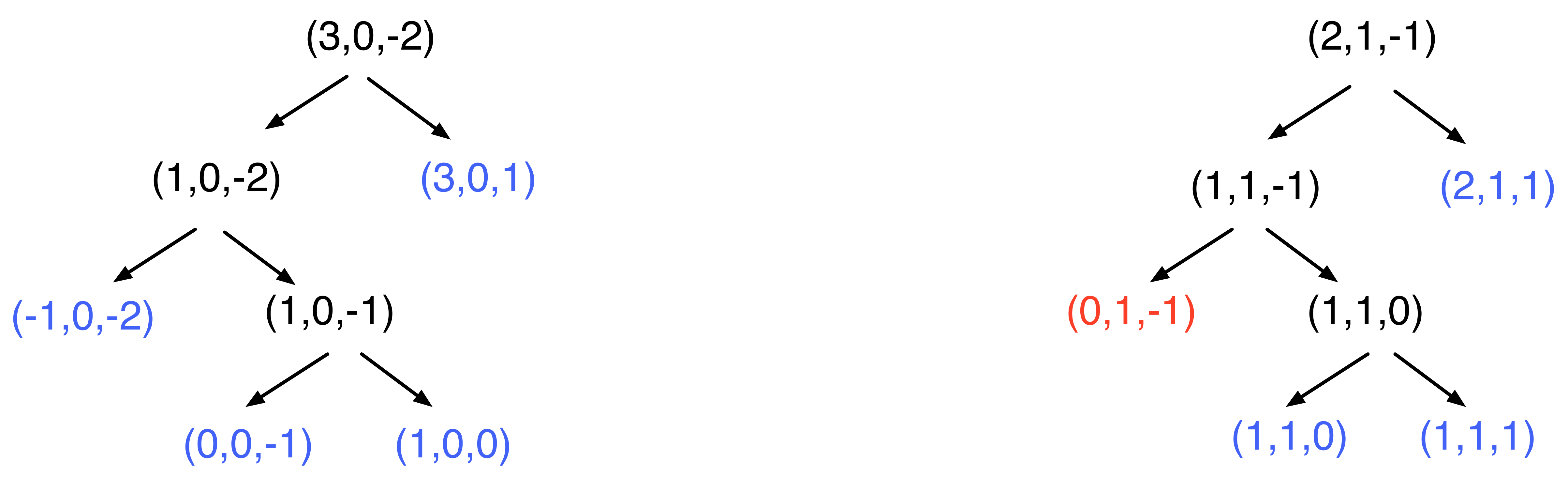}
\caption{On the left: the algorithm applied to \(\chi_1=(3,0,-2)\). On the right: the same steps  applied   to \(\chi_2=(2,1,-1)\). The vector \((0,1,-1)\) (in red) is  not `good'.}
\label{algoritmo1}

\end{figure}

Now we apply the algorithm  one more time and  obtain the two vectors \((0,0,-1),(0,1,0)\) whose coordinates are respectively all nonpositive and all nonnegative.    Figure \ref{algoritmo2}  shows the final output for the coordinate of $\chi_1$ (left hand side) and of $\chi_2$ (right hand side). \vskip20pt

\begin{figure}[h!]
\centering
\includegraphics[scale=0.2]{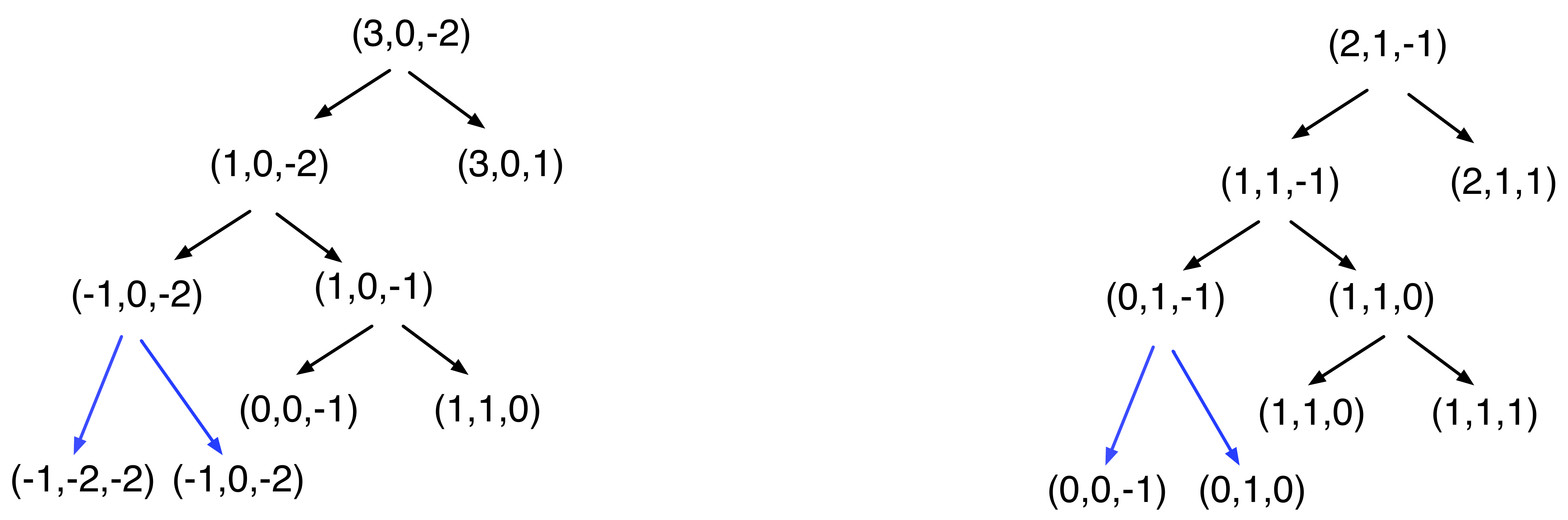}
\caption{The algorithm of Figure \ref{algoritmo1} is completed by a further step (blue arrows) since the vector \((0,1,-1)\) was not `good' .}
\label{algoritmo2}

\end{figure}

Figure \ref{algoritmo3} shows that after applying  the steps of the algorithm, in the end we subdivide the cone 
\(\{e^1,e^2,e^3\}\) into the following maximal cones:
\(\sigma_1=\{e^1,e^2,e^1+e^3  \}\),  \(\sigma_2=\{e^1+2e^3,e^2+e^3,e^3  \}\), \(\sigma^3=\{e^1+2e^3,e^2,e^2+e^3  \}\), \(\sigma^4=\{2e^1+3e^3,e^2,e^1+2e^3  \}\), \(\sigma^5=\{e^1+e^3,e^2,2e^1+3e^3 \}\).

\begin{figure}[h!]
\centering
\includegraphics[scale=0.3]{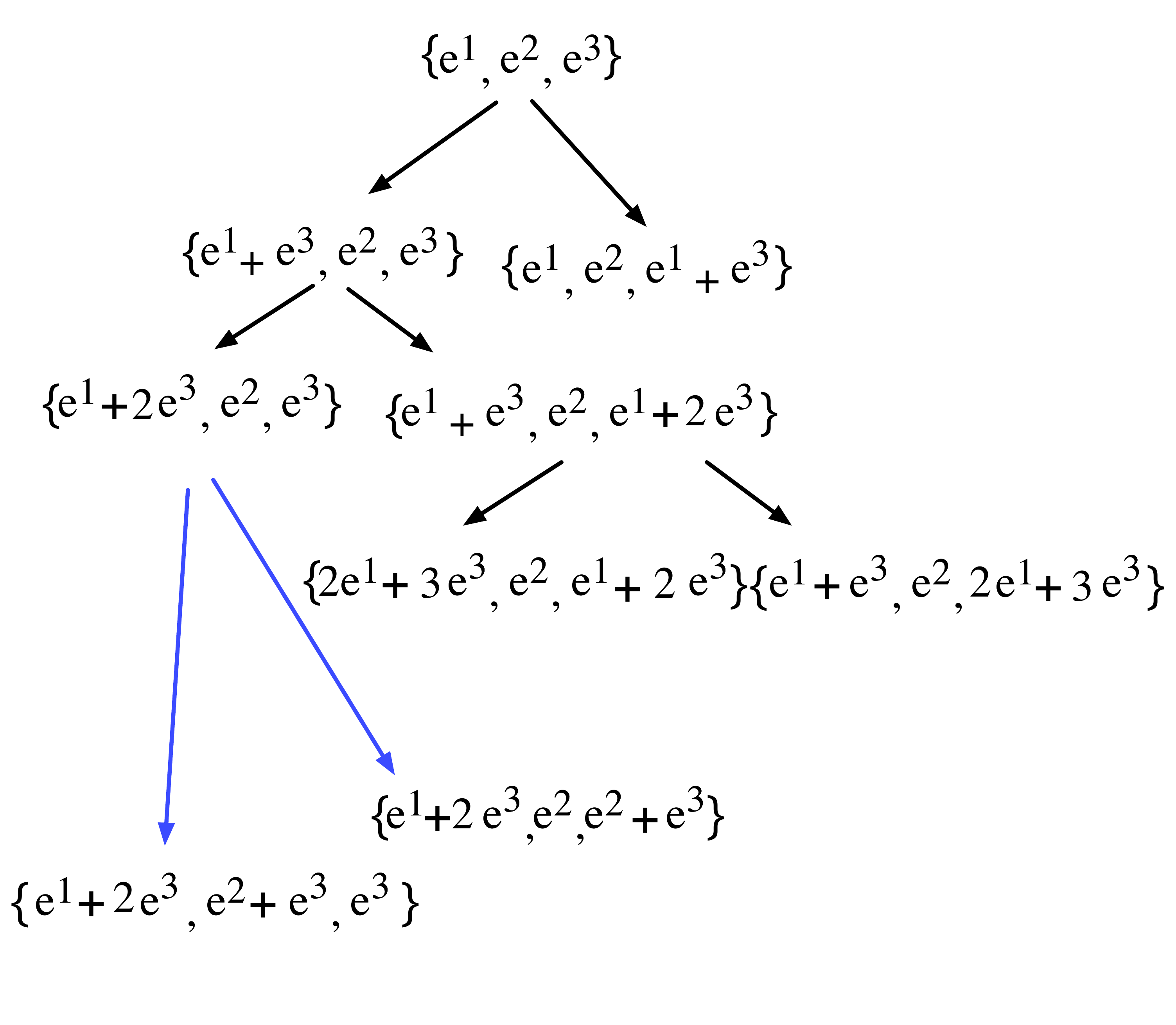}
\caption{The subdivision of the cone \(\{e^1,e^2,e^3\}\)  produced by the steps of the  algorithm  of Figure \ref{algoritmo2}.}
\label{algoritmo3}

\end{figure}

We now  give  an  example of the algorithm applied to a 2-dimensional complete fan. We let $T$ be 2-dimensional and choose   a basis of $X^*(T)$. The starting fan is then the one  whose maximal dimensional cones are the four  quadrants with respect to the chosen basis. We then take $\Xi=
\{\chi_1=(1,0), \chi_2=(1,2)\}$. We remark that the algorithm needs to be applied only in the second and fourth quadrants. The reader can  easily check that the final output is the fan given in Figure \ref{belladecasa}.

\begin{figure}[h!]
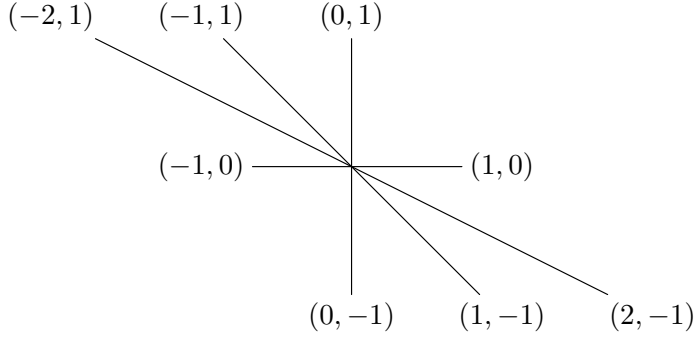

\label{belladecasa}
\qquad\qquad\qquad \qquad \xy <1cm,0cm>:
% (0,0)*=0{+}="+" ;
% (2,1)*=0{\times}="*" **@{.} ,
% (1,0)*+{} ; (2,2)*+{} **@{-},
(-2,0)*+{(-1,0)} ; (2,0)*+{(1,0)} **@{-},
(0,-2)*+{(0,-1)} ; (0,2)*+{(0,1)} **@{-},
(2,-2)*+{(1,-1)} ; (-2,2)*+{(-1,1)} **@{-},
(-4,2)*+{(-2,1)} ; (4,-2)*+{(2,-1)} **@{-}
% ?!{"+";"*"} *{\bullet}
\endxy\caption{Fan for $(1,0), (1,2).$}
\end{figure}

\section{The construction of the  toric variety \(K_{\Delta(\A)}\)}
\label{sec:strategia}

As we mentioned in the Introduction, given a toric arrangement \(\A\) in \(T\)  the main step in our  construction  of a projective wonderful model for the complement \(\emme(\A)\) is  the construction of a smooth projective toric variety \(K_{\Delta(\A)}\) (where \(\Delta(\A)\) denotes its fan), containing  \(T\) as a dense open set.

We will describe \(K_{\Delta(\A)}\) by describing its fan \(\Delta(\A)\); this in turn  will be obtained by a repeated application of the algorithm of Theorem \ref{teoalgoritmooriginale}.

 As a first step we  choose a basis for the lattice $X^*(T)$. This gives an isomorphism of $T$ with $(\C^*)^n$, an isomorphism of $X^*(T)$ with $\mathbb Z^n$ and of $X_*(T)\otimes \mathbb R$ with $\R^n$. The decomposition of $\R^n$ into quadrants gives a fan $\Delta$ whose associated $T$ variety is isomorphic to  \((\pp^1)^n\).

Let us now consider the toric arrangement  \(\A=\{\mathcal K_{1},...,\mathcal K_{m}\}\), where \(\mathcal K_i=\mathcal K_{\Gamma_i, \phi_i}\), and,  for every \(i=1,...,m\), let \(\chi_{i1},...,\chi_{ij_i}\) be an integral basis of \(\Gamma_i\subset X^*(T)\).  Let \(\Xi\)  be the set whose elements are the characters  \(\chi_{is}\), for every \(i=1,...,m\) and for every \(1\leq s\leq j_i\).

By applying Theorem \ref{teoalgoritmooriginale}  to the fan $\Delta$ and to the set  \(\Xi\subset X^*(T)\), we obtain a new fan $\Delta(\mathcal A)$ such that each character $\chi_{i,s}$ has property $(E)$ with respect to $\Delta(\mathcal A)$. 

\begin{defin} A layer $\mathcal K_{\Gamma,\phi}$ is a layer {\em for} the arrangement \(\A=\{\mathcal K_{1},...,\mathcal K_{m}\}\), or a \(\A\)-layer, if it is a connected component of the intersection of some of the $\mathcal K_i$. \end{defin}
%\end{document}
We have thus proved
\begin{prop}\label{lacostru} Let $\mathcal A$ be a toric arrangement. Choose a basis for $X^*(T)$ and let $T\subset (\pp^1)^n$ be the corresponding $T$ embedding. 
%\todo{qui tornerebbe comodo per il seguito  scrivere direttamente \(K_{\Delta(\A)}\) e non \(Y\)}
%There is a projective $T$. 
There is a fan $\Delta(\mathcal A)$ such that  
\begin{enumerate}\item[1)]   The $T$ embedding $K_{\Delta(\A)}$ is smooth and it is obtained from $(\pp^1)^n$ by a sequence of blow ups along closures of orbits of codimension 2.
\item[2)] Every \(\A\)-layer  has property $(E)$ with respect to $\Delta(\mathcal A).$\end{enumerate} \end{prop}

A few observations are in order:
\begin{rem} The construction of $\Delta(\mathcal A)$ strongly depends on
\begin{enumerate}\item The choice of a basis for $X^*(T)$.
\item The choice of the set of characters $\Xi$.
\item The strategy in which our algorithms is implemented.
\end{enumerate}
It is desirable to understand whether and how one could develop a more efficient procedure.
We observe that, as a geometric counterpart to the combinatorial blowups of fans, we have   that the toric variety \(K_{\Delta(\A)}\)  is obtained from  \((\pp^1)^n\) by a sequence of blowups: each blowup is the blowup of a toric T-variety  along the closure  of a  2-codimensional  T-orbit.
\end{rem}
\begin{rem}
\label{rem:spancompleto}
Let us consider the linear span \(V'\) in \(V=X^*(T)\otimes \R\) of the vectors \(\chi_{is}\) mentioned above (so \(i=1,..,m\) and \(1\leq s\leq j_i\)). If \(V' \neq X^*(T)\otimes \R\) we can choose a basis \(\eta_1,...,\eta_n\) of \(X^*(T)\) such that \(\eta_1,...,\eta_r)\) span $V'$ (from the computational point of view  it could be useful to pick as many   \(\eta_i\) as possible from  the set \(\{\chi_{is}\}\)). We have an isomorphism  \(T=(\C^*)^n\) as \((\C^*)^r\times (\C^*)^{(n-r)}\) and one easily sees that our problem reduces to finding a smooth projective toric variety for the toric arrangement \(\A\) restricted to \((\C^*)^r\).  So without loss of generality in the sequel we will always suppose \(V'=V=X^*(T)\otimes \R\).
\end{rem}

\section{The arrangement of subvarieties \(\elle\)}
\label{sec:jacobian}
%The construction of the toric variety  \(K_{\Delta(\A)}\) completes Step \(1)\) of our strategy.
Given the toric variety \(K_{\Delta(\A)}\)  constructed in Section \ref{sec:strategia}, we  denote by \({\mathcal Q}\) the set whose elements are  the subvarieties \({\overline {\mathcal K}_{i}}\) and the irreducible components \(D_\alpha\) of the complement \(K_{\Delta(\A)}- T\). We then denote by \(\elle\) the poset  made by all the connected components of all the intersections of some of the  elements of \({\mathcal Q}\).

\begin{teo} 
\label{sec:teosubvarieties}
The family \(\elle\) is an arrangement of subvarieties according to Definition \ref{def:nonsimple}.

\end{teo}
\begin{proof}
Let us consider an  element \(S\in \elle\),  that is a connected component of the intersection \({\widetilde S}\) of some of the  elements in  \(\mathcal Q\).  If all of these elements are irreducible  components \(D_\alpha\) of the complement \(K_{\Delta(\A)}- T\), from the theory of toric varieties we  known that \(S\) is smooth and that the intersection is clean.

Let us then consider the case when  \({\widetilde S}\) is  the intersection of  the closures of some  layers of the arrangement \(\A=\{\mathcal K_{1},...,\mathcal K_{m}\}\),  say   $\mathcal K_{1}$, $\mathcal K_{2}$,..., $\mathcal K_{s}$. Therefore \(S\) is the closure of a \(\A\)-layer   $\mathcal K_{\Gamma,\phi}$.

By point 2 of Proposition \ref{lacostru}, 
\( \mathcal K_{\Gamma,\phi}\) has property $(E)$ with respect to $\Delta(\mathcal A)$ and it then follows from point 1 of Theorem \ref{cllay} that \(S\) is a smooth toric variety. By the description of point 1 of Theorem \ref{cllay} it also follows that if we further intersect \(S\) with some irreducible  components \(D_\alpha\) of the complement \(K_{\Delta(\A)}- T\), we get that the resulting connected components are boundary components of the toric variety \(S\), and therefore they are smooth.

It remains to prove that the intersection of two strata \(\Lambda_1, \Lambda_2\) in \(\elle\), if it is not empty,  satisfies the condition on the tangent space, i.e., 
\[T_{\Lambda_i\cap \Lambda_j,y} = T_{\Lambda_i,y}\cap T_{\Lambda_j,y}\]
for every \(y\in  \Lambda_i\cap \Lambda_j\).

The inclusion \[T_{\Lambda_i\cap \Lambda_j,y} \subseteq T_{\Lambda_i,y}\cap T_{\Lambda_j,y}\]
is obvious, then it is sufficient to check that the dimensions are the same.
We have already proved that \(\Lambda_i\cap \Lambda_j\) is smooth, so \(dim \  T_{\Lambda_i\cap \Lambda_j,y}= dim \ \Lambda_i\cap \Lambda_j\).

Again,  let us first consider the  case is when \(\Lambda_i\) and \(\Lambda_j\) are connected components of the intersection of  the closures of some  layers of the arrangement \(\A\). Therefore we can put \(\Lambda_i= \overline {\mathcal K}_{\Gamma_i,\phi_i}\),  \(\Lambda_j= \overline {\mathcal K}_{\Gamma_j,\phi_j}\).
 
 Then every  connected component of \(\Lambda_i\cap \Lambda_j\) is of the form \(\overline {\mathcal K}_{\overline \Gamma,\overline \phi}\), where \(\overline \Gamma\) is the saturation of the lattice \(\Gamma_1+\Gamma_2\).
 
 In the  proof of point 1 of Theorem \ref{cllay} we showed, by a local computation in a  chart of \(K_{\Delta(\A)}\),  that the rank of the Jacobian matrix of the equations defining \(\overline {\mathcal K}_{\overline \Gamma,\overline \phi}\) is equal to the rank of \(\overline \Gamma\).
 Therefore  the dimension of \(\overline {\mathcal K}_{\overline \Gamma,\overline \phi}\) is equal to \(n- rank \ \overline \Gamma\).

Now we observe that the dimension of  \(T_{\Lambda_i,y}\cap T_{\Lambda_j,y}\)  is equal to the dimension of the intersection of the kernels of the Jacobian matrices of the equations defining \(\overline {\mathcal K}_{ \Gamma_i, \phi_i}\) and \(\overline { \mathcal K}_{ \Gamma_j, \phi_j}\). This dimension,  as one can immediately check,   is equal to \(n- rank \ (\Gamma_1+\Gamma_2)\). 
Since \(rank \ \overline \Gamma = rank \ (\Gamma_1+\Gamma_2)\) this concludes the proof in this case.

Let us now consider the  case when \(\Lambda_i\) (or \(\Lambda_j\)) is equal to \(\overline {\mathcal K}_{\Gamma_i,\phi_i}\) intersected with some components \(D_\alpha\) of the complement \(K_{\Delta(\A)}- T\). The relevant remark is that in a local chart  a component  \(D_\alpha\) has an equation of type \(x_\nu=0\), therefore if the intersection  \(\Lambda_i\cap \Lambda_j\) is not empty  the variable \(x_\nu\) does not appear in the equations that define \(\overline {\mathcal K}_{\Gamma_i,\phi_i}\) and \(\overline {\mathcal K}_{\Gamma_j,\phi_j}\).

Up to this, the   computation of the dimensions of \(T_{\Lambda_i\cap \Lambda_j,y}\) and  \(T_{\Lambda_i,y}\cap T_{\Lambda_j,y}\) is then completely similar to the one of the preceding case.

\end{proof}

\section{Root systems and related examples}\label{sec:examples}
It is important to point out that our proof of Theorem \ref{sec:teosubvarieties} shows that, given a toric  arrangement $\mathcal A=\{\mathcal K_1,\ldots ,\mathcal K_m\}$ and a smooth complete fan $\Theta$, in order for  the family \(\elle\) consisting of all connected components of intersections of some of the $\overline {\mathcal K}_i$ and some components of the complement $K_\Theta\setminus T$, to be an arrangement of subvarieties it suffices that each  of the $\mathcal K_i$'s has property $(E)$ with respect to $\Theta$.

This fact allows us to give a class of examples for which we do not have to go through the complicated algorithm of Section \ref{secalgoritmooriginale}.

We first notice that Theorem \ref{cllay} provides another point of view on our construction of the toric variety in the case of a divisorial arrangement. 
Let  us consider the divisorial toric arrangement 
\(\A=\{\mathcal K_{\chi_1,b_1},...,\mathcal K_{\chi_m,b_m}\}\) in \(T\), where $\Gamma_i=\Z \chi_i$, $\chi_i$ a primitive character.
% after the usual identification of \(X*(T)\) with \(\Z^n\),  \(\Gamma_i\) is generated by the primitive character \(\chi_i=(a_{i1},a_{i2},....,a_{in})\in \Z^n\) (i.e. the \(MCD\) of the numbers \(a_{i1},a_{i2},....,a_{in}\) is equal to 1)  and \(\phi_i(\chi_i)=b_i\in \C^*\).  
In   \(V=hom_\Z(X^*(T),\R)\) take the real hyperplane arrangement  \(\mathcal H_\A=\{H_{\chi_1},...,H_{\chi_m}\}\) of the hyperplanes orthogonal to the $\chi_i$'s.
 The chambers of this hyperplane arrangement   define some \(n\)-dimensional rational polyhedral cones, {which  we can  assume   to be strongly convex (see Remark \ref{rem:spancompleto})}. Taking all non empty   intersections of   (the closures of )   these chambers, we obtain a complete fan \(\Phi\), that is not necessarily smooth:  as a consequence of Theorem \ref{cllay}  (see Remark \ref{rem:strati}) we have that the fan \(\Delta(\A)\) provided by our algorithm gives  a particular  subdivision of this fan but   any smooth complete fan subdividing \(\Phi\) would do.

If it happens that \(\Phi\) is  already a smooth projective fan then there is no need to apply our algorithm so that
 the toric variety \(K_{\Phi}\)  gives a canonical choice for our construction.
 
 Here is the main example of this situation. Suppose  $T$ is the maximal torus in an adjoint semisimple group $G$  and $R\subset X^*(T)$ is the corresponding  root system. We choose a set of positive roots $R^+$ and  fix for each $\alpha\in R^+$ a constant $b_\alpha\in \C^*$ We then get the toric arrangement $\mathcal A=\{\mathcal K_{\alpha,b_\alpha}\}_{\alpha\in R^+}$. It is immediate from the definition that the corresponding fan $\Phi$ in  \(V=hom_\Z(X^*(T),\R)\) is given by the Weyl chambers and their faces. Also each Weyl chamber corresponds to a choice of a basis of simple roots and every root is expressed as a linear combination with respect to  such a basis with all non negative or non positive coefficients. 
 
 If we then take the family \(\elle\) of the connected components of all the intersections of the closures of $ \overline {\mathcal K}_{\alpha,b_\alpha}$ and of boundary divisors in $K_\Phi$
  we get \begin{prop}
The family \(\elle\) is an arrangement of subvarieties in $K_\Phi$. \end{prop}

\begin{rem} 1) The variety $K_\Phi$ appears in various relevant instances, 
for example as the closure of a ``generic" $T$ orbit in the flag variety, or as the closure of $T$ in the wonderful compactification of $G$. 

2) If $W$ denotes the Weyl group of the root system $R$, $W$ acts on the embedding $K_\Phi$ compatibly with its action on $T$. Now, if  for a negative root $\alpha$, we set 
$b_{\alpha}=b_{-\alpha}^{-1}$, we obtain a map $R\to \C^*$. If this map is constant on $W$-orbits then $W$ also acts on $\mathcal A$ and on the family \(\elle\). So taking a building set stable under the $W$ action we obtain a $W$ equivariant compactification of $\mathcal A$.
\end{rem}

%Therefore \(K_{\Gamma_\Phi}\)  is a `canonical' choice of a toric variety associated to \(\A_\Phi\).
Notice that obviously the embedding $K_\Phi$ works as well for any arrangement $\mathcal A'\subset \mathcal A$.

For instance, given a directed graph \(\Gamma\), one can associate to its vertices \(\gamma_1,...,\gamma_{n+1}\) the vectors \(e_1,...,e_{n+1}\) of a basis of \(\Z^n\), and to its arrows their incidence vectors (if an arrow connects \(\gamma_i\) and \(\gamma_j\) and points to \(\gamma_j\) we associate to it  the vector \(e_i-e_j\)).
If we think the root system of type $A_n$ as the set of vectors $\alpha_{i,j}=e_i-e_j$, where $i,j=1,\ldots n+1$ and \(i\neq j\), then to such a directed graph it is associated  the subset of $\mathcal A$ (for $A_n$) consisting of those $\mathcal K_{\alpha,b_{\alpha}}$ for which $\alpha=e_i-e_j$ comes from an arrow of our graph.

\section{A simple remark on  the integer cohomology and on  the Chow ring of a projective  model}
Let us consider a toric arrangement \(\A\)  and denote by \(\mathcal W (\A)\) any  projective wonderful model for \(\A\) constructed according to the strategy described in this paper. We will prove that   the integer cohomology of \(\mathcal W (\A)\)  is even and torsion free and that  the integer cohomology ring   is isomorphic to the Chow ring. 
%In the next Section \ref{cohomologyring} we will describe more in detail this ring.

Let us   start by   recalling  from \cite{DCLP} the definition of  property  \((S)\) for a smooth projective algebraic variety.
If X is an smooth and projective algebraic variety,  let us denote  by $A^k(X)$ the group generated by the $k$-codimensional irreducible subvarieties modulo rational equivalence (see \cite{Fultonintersection} 1.3) and by \(A^*(X)\) the Chow ring.

Let $H^j(X)$ be the  integer  cohomology of X. There is a canonical ring homomorphism (see for instance \cite{Fultonintersection} 19.1 or \cite{Coxlittleschenck} 12.5):
$$\Phi : A^*(X)\to   H^*(X)$$  that sends 
$ A^j(X)\to   H^{2j}(X)$ for every \(j\).

The following definition is adapted from  \cite{DCLP} (we are specializing to our case where  Poincar\'e duality holds). 

 \begin{defin} A smooth and projective algebraic variety X is said to have property \((S)\) if
\begin{enumerate} \item  $H^i(X) =0$ for $i$ odd and  $H^j(X)$ has no torsion for even \(j\). \item $\Phi_{|A^j}: A^j(X) \mapsto  H^{2j}(X)$ is an isomorphism  for all $j\geq 0.$\end{enumerate}
\end{defin}
In particular, if  a smooth  projective algebraic variety X   satisfies  property \((S)\) we have that \(\Phi\) gives a ring isomorphism  \(A^*(X)\cong H^*(X)\).

\begin{teo} The projective wonderful variety  \(\mathcal W (\A)\) has property  \((S)\).
\end{teo}
\begin {proof} 
We start by remarking that if we have two smooth complete subvarieties $Y\subset X$ such 
that both $Y$ and $X$ have property  \((S)\), then also the blowup $\widetilde X={Bl}_YX$ of $X$ along $Y$ has property  \((S)\). 
Indeed   recall (see for instance Theorem 15.11 in \cite{eisenbudharris}) that, setting $E$ equal to the exceptional divisor, we have the exact sequence of  Chow groups   $$   0\rightarrow   A(Y)\rightarrow  A(X)\oplus A(E)  \rightarrow A(\widetilde X) \rightarrow 0$$ 
Since $E$ is a projective bundle over $Y$, then $E$ has property  \((S)\). Also, since $Y$, $X$ and $E$ have no odd cohomology,  we get,  by comparing the exact sequence above with the 
 corresponding  sequence for cohomology, that also  \(\widetilde X\) has the property \((S)\).

This allows us to prove inductively that  the projective model \(\mathcal W (\A)\) has the property \((S)\).
We start by observing that  \(\mathcal W (\A)\)  is constructed by the blowup process described by MacPherson-Procesi and Li (see Theorem  \ref{teo:listabuilding}) starting from the smooth projective toric variety   \(K_{\Delta(\A)}\).
Now from the theory of toric varieties we know that a smooth projective toric variety has the property \((S)\) (see for instance \cite{Coxlittleschenck} 12.5).

As we noticed  in Remark \ref{rem:strati}, from  Theorem \ref{cllay} and from the standard theory of toric varieties  we know that also all the strata in  $\mathcal L$ are smooth projective toric varieties.

So in the first step of the construction  we blow up a smooth projective toric variety along a stratum  that is isomorphic to a smooth projective toric variety. The  resulting  variety has property \((S)\) and also the proper transforms of the other strata have property \((S)\), since  (again by    Theorem \ref{cllay} and  standard theory of toric varieties) they  are blowups of smooth projective toric varieties along smooth projective toric subvarieties.   

By induction on the dimension one can immediately see  that  at every step of the blowup process we blow up a variety that has property \((S)\) along a subvariety that has property \((S)\). 
\end{proof}

\addcontentsline{toc}{section}{References}
%\nocite{*}
\bibliographystyle{acm}
\bibliography{Bibliogpre} 
\end{document}